\documentclass[reqno, 11pt]{amsart}
\usepackage{amsmath}
\usepackage{amssymb}
\usepackage{amsthm}
\usepackage{times}
\usepackage{latexsym}
\usepackage[mathscr]{eucal}
\usepackage[centertags]{amsmath}
\usepackage{amsmath}
\usepackage{amsthm}
\usepackage{color}
\usepackage[english]{babel}
\usepackage[latin1]{inputenc}
\usepackage[T1]{fontenc}
\usepackage{graphicx}
\usepackage{amsfonts}
\usepackage{amssymb}
\usepackage[all,cmtip]{xy}

\numberwithin{equation}{section}

\newtheorem{theorem}{Theorem}[section]

\newtheorem{definition}[theorem]{Definition}

\email{abdoulsalam.diallo@uadb.edu.sn}
\email{punam\_2101@yahoo.co.in}
\thanks{}

\begin{document}
\title[Four-dimensional semi-Riemannian Szab\'o manifolds]{Four-dimensional semi-Riemannian Szab\'o manifolds}
\author[Abdoul Salam Diallo and Punam Gupta]{Abdoul Salam Diallo and Punam
Gupta}
\address{\llap{\,} Universit\'e Alioune Diop de Bambey\\
\indent
UFR SATIC, D\'epartement de Math\'ematiques\\
\indent
Equipe de Recherche en Analyse Non Lin\'eaire et G\'eom\'etrie\\
\indent
B. P. 30, Bambey, S\'en\'egal}
\address{\llap{**\,} Department of Mathematics \& Statistics\\
\indent
 School of Mathematical \& Physical Sciences\\
\indent
Dr. Harisingh Gour University\\
\indent
Sagar-470003, Madhya Pradesh}

\begin{abstract}
In this paper, we prove that the deformed Riemannian extension of
any affine Szab\'o manifold is a Szab\'o pseudo-Riemannian metric
and vice-versa. We proved that the Ricci tensor of an affine surface
is skew-symmetric and nonzero everywhere if and only if the affine
surface is Szab\'o. We also find the necessary and sufficient
condition for the affine Szab\'o surface to be recurrent. We prove
that for an affine Szabó recurrent surface the recurrence covector
of a recurrence tensor is not locally a gradient.
\end{abstract}

\maketitle

\section{Introduction}

Let $T^{\ast }M$ be the cotangent bundle of $n$-dimensional manifold $M$
with a torsion free affine connection $\nabla $. Patterson and Walker \cite%
{PattersonWalker1952} introduced the notion of Riemannian extensions and
showed how to construct a pseudo-Riemannian metric on the $2n$-dimensional
cotangent bundle of any $n$-dimensional manifold with a torsion free
connection. Afifi \cite{Afifi1954} studied the local properties of
Riemannian extension of connected affine spaces. Riemannian extensions were
also studied by Garcia-Rio et al. \cite{Garcia1999} for Osserman manifolds.
One of the author Diallo \cite{Diallo2013} find the fruitful results for the
Riemannian extension of an affine Osserman connection on $3$-dimensional
manifolds. In \cite{Garcia}%
, the authors generalized the Riemannian extension to the deformed
Riemannian extensions. In the recent paper \cite{DialloGupta2020},
we construct example of pseudo-Riemannian Szabó metrics of signature
$(2,2)$ by using the deformed Riemannian extension, whose Szabó
operators are nilpotent. The Riemannian extension can be constructed
with the help of the coefficients of the
torsion free affine connection. For Riemannian extensions, also see \cite%
{KowalskiSekizawa2011, KowalskiSekizawa2014, Vanhecke1982}. For deformed Riemannian extensions, also see \cite%
{Bonome1998,Brozos2009, DialloHassirouKatambe2013}.

In this paper, we study the deformed Riemannian extensions of affine
Szab\'o manifold. Our paper is organized as follows. In the section
\ref{Deformed}, we recall some basic definitions and results on the
deformed Riemannian extension. In section \ref{Szabo}, we provide
some known results on affine Szabó manifolds. We proved that the
Ricci tensor of an affine surface is skew-symmetric and nonzero
everywhere if and only if affine surface is Szab\'o. We also find
the necessary and sufficient condition for the affine Szab\'o
surface to be recurrent. We prove that for an affine Szab\'o
recurrent surface the recurrence covector of a recurrence tensor is
not locally a gradient. Finally in section \ref{Main}, we prove that
the deformed Riemannian extension of any affine Szab\'o manifold is
a Szabó pseudo-Riemannian metric and vice-versa.

Throughout this paper, all manifolds, tensors fields and connections are
always assumed to be $\mathcal{C}^{\infty }$-differentiable.

\section{Deformed Riemannian extensions}

\label{Deformed}

Let $T^{\ast }M$ be the cotangent bundle of $n$-dimensional affine manifold $%
M$ with torsion free affine connection $\nabla $ and let $\pi :T^{\ast
}M\rightarrow M$ be the natural projection defined by
\begin{equation*}
\pi (p,\omega )=p\in M\quad \mathrm{and}\quad (p,\omega )\in T^{\ast }M.
\end{equation*}
A system of local coordinates $(U,u_{i}),i=1,\ldots ,n$ around $p\in M$
induces a system of local coordinates $(\pi ^{-1}(U),u_{i},u_{i^{\prime
}}=\omega _{i}),i^{\prime }=n+i=n+1,\ldots ,2n$ around $(p,\omega )\in
T^{\ast }M$, where $\omega _{i}$ are components of covectors $\omega $ in
each cotangent space $T_{p}^{\ast }M$, $p\in U$ with respect to the natural
coframe $\{du^{i}\}$. Let $\partial _{i}=\dfrac{\partial }{\partial u_{i}}$
and $\partial _{i^{\prime }}=\dfrac{\partial }{\partial \omega _{i}}%
,i=i,\ldots ,n$, then at each point $(p,\omega )\in T^{\ast }M$,
\begin{equation*}
\Big\{(\partial _{1})_{(p,\omega )},\ldots ,(\partial _{n})_{(p,\omega
)},(\partial _{1^{\prime }})_{(p,\omega )},\ldots ,(\partial _{n^{\prime
}})_{(p,\omega )}\Big\},
\end{equation*}%
is a basis for the cotangent space $(T^{\ast }M)_{(p,\omega )}$. For more
details on the geometry of cotangent bundle, see \cite{Yano1972}.

The Riemannian extension $g_{\nabla }$ is the pseudo-Riemannian metric on $%
T^{\ast }M$ of neutral signature $(n,n)$ characterized by the identity \cite%
{Garcia}
\begin{equation*}
g_{\nabla }(X^{C},Y^{C})=-\iota (\nabla _{X}Y+\nabla _{Y}X),
\end{equation*}%
where $X^{C}$ is a complete lift of the vector field $X$ on $M$ and the
function $\iota X:T^{\ast }M\longrightarrow \mathbb{R}$ defined by
\begin{equation*}
\iota X(p,\omega )=\omega (X_{p}).
\end{equation*}%
For more details, see \cite{Garcia}. In the locally induced coordinates $%
(u_{i},u_{i^{\prime }})$ on $\pi ^{-1}(U)\subset T^{\ast }M$, the Riemannian
extension \cite{PattersonWalker1952} is expressed by
\begin{equation*}
g_{\nabla }=\left(
\begin{array}{cc}
-2u_{k^{\prime }}\Gamma _{ij}^{k} & \delta _{i}^{j} \\
\delta _{i}^{j} & 0%
\end{array}%
\right) ,
\end{equation*}%
with respect to the basis $\{\partial _{1},\ldots ,\partial _{n},\partial
_{1^{\prime }},\ldots ,\partial _{n^{\prime }}\}(i,j,k=1,\ldots ,n;k^{\prime
}=k+n)$, where $\Gamma _{ij}^{k}$ are the coefficients of the torsion free
affine connection $\nabla $ with respect to $(U,u_{i})$ on $M$.

Riemannian extensions provide a link between affine and pseudo-Riemannian
geometries, therefore by using the properties of the Riemannian extension $%
g_{\nabla }$, we investigate the properties of the affine connection $\nabla
$. Like, $(M,\nabla )$ is locally symmetric if and only if $(T^{\ast
}M,g_{\nabla })$ is locally symmetric. In the same way, $(M,\nabla )$ is
projectively flat if and only if $(T^{\ast }M,g_{\nabla })$ is locally
conformally flat \cite{Calvino2010}.

Let $\phi $ be a symmetric $(0,2)$-tensor field on an affine manifold $%
(M,\nabla )$. In \cite{Calvino2010}, the authors introduced a deformation of
the Riemannian extension by means of a symmetric $(0,2)$-tensor field $\phi $
on $M$. They considered the cotangent bundle $T^{\ast }M$ equipped with the
metric $g_{\nabla }+\pi ^{\ast }\phi $, which is called the deformed
Riemannian extension.

The deformed Riemannian extension denoted $g_{(\nabla ,\phi )}$ is the
metric of neutral signature $(n,n)$ on the cotangent bundle given by
\begin{equation*}
g_{(\nabla ,\phi )}=g_{\nabla }+\pi ^{\ast }\phi .
\end{equation*}%
In local coordinates the deformed Riemannian extension is given by
\begin{equation*}
g_{(\nabla ,\phi )}=\left(
\begin{array}{cc}
\phi _{ij}(u)-2u_{k^{\prime }}\Gamma _{ij}^{k} & \delta _{i}^{j} \\
\delta _{i}^{j} & 0%
\end{array}%
\right) ,
\end{equation*}%
with respect to the basis$\{\partial _{1},\ldots ,\partial _{n},\partial
_{1^{\prime }},\ldots ,\partial _{n^{\prime }}\},(i,j,k=1,\ldots
,n;k^{\prime }=k+n)$, where $\Gamma _{ij}^{k}$ are the coefficients of the
torsion free affine connection $\nabla $ and $\phi _{ij}$ are the local
components of the symmetric $(0,2)$-tensor field $\phi $. Equivalently,
\begin{equation*}
g_{(\nabla ,\phi )}(\partial _{i},\partial _{j})=\phi _{ij}(u)-2u_{k^{\prime
}}\Gamma _{ij}^{k};\quad g_{(\nabla ,\phi )}(\partial _{i},\partial
_{j^{\prime }})=\delta _{i}^{j};\quad g_{(\nabla ,\phi )}(\partial
_{i^{\prime }},\partial _{j^{\prime }})=0
\end{equation*}%
Note that the crucial terms $g_{(\nabla ,\phi )}(\partial _{i},\partial
_{j}) $ now no longer vanish on the $0$-section, which was the case for the
Riemannian extension, the Walker distribution is the kernel of the
projection from $T^{\ast }M$:
\begin{equation*}
\mathcal{D}=\ker \{\pi ^{\ast }\}=\mathrm{Span}\{\partial _{i^{\prime }}\}.
\end{equation*}%

In the deformed Riemannian extension, the tensor $\phi $ plays an important
role. If the underlying connection is flat, the deformed Riemannian
extension need not be flat \cite{Garcia}. Deformed Riemannian extensions
have nilpotent Ricci operator therefore they are Einstein if and only if
they are Ricci flat. So deformed Riemannian extension can be used to
construct non-flat Ricci flat pseudo-Riemannian manifolds \cite{Calvino2010}.

\section{The affine Szabó manifolds}\label{Szabo}

Let $(M,\nabla )$ be an affine manifold and $X\in \Gamma (T_{p}M)$. The
affine Szabó operator $\mathcal{S}^{\nabla }(X)$ \cite{Szabo} with respect
to $X$ is a function from $T_{p}M$ to $T_{p}M$, $p\in M$ defined by%
\begin{equation*}
\mathcal{S}^{\nabla }(X)Y=(\nabla _{X}\mathcal{R}^{\nabla })(Y,X)X,
\end{equation*}%
for any vector field $Y$ and where $\mathcal{R}^{\nabla }$ is the curvature
operator of the affine connection $\nabla$. The affine Szabó operator
satisfies $\mathcal{S}^{\nabla }(X)X=0$ and
$\mathcal{S}^{\nabla }(\beta X)=\beta ^{3}\mathcal{S}(X)$, for
$\beta \in \mathbb{R}^{\ast }$. If $Y=\partial _{m}$, for
$m=1,2,\ldots ,n$ and $X=\sum_{i}\alpha _{i}\partial _{i}$, we have
\begin{equation*}
\mathcal{S}^{\nabla }(X)\partial _{m}=\sum_{i,j,k=1}^{n}\alpha _{i}\alpha
_{j}\alpha _{k}(\nabla _{i}\mathcal{R}^{\nabla })(\partial _{m},\partial
_{j})\partial _{k},
\end{equation*}%
where $\nabla _{i}=\nabla _{\partial _{i}}$.

Let $(M,\nabla )$ be an affine manifold and $p\in M$. $(M,\nabla )$ is said
to affine Szabó at $p\in M$ if the affine Szabó operator $\mathcal{S}%
^{\nabla }$ has the same characteristic polynomial for every vector field $X$
on $M$. If $(M,\nabla )$ is affine Szabó at each $p\in M$, then $(M,\nabla )$
is known as affine Szabó. For more details, see \cite%
{DialloLongwapMassamba2017-1}.

Now, we give a known result for later use.

\begin{theorem}\cite{DialloMassamba2017}\label{TheorSZA1} Let $(M,\nabla)$
be an $n$-dimensional affine manifold and $p\in M$. Then $(M,\nabla)$
is affine Szabó at $p\in M$ if and only if the characteristic polynomial
of the affine Szabó operator $\mathcal{S}^{\nabla }$ is
$P_{\lambda }[\mathcal{S}^{\nabla}(X)]=\lambda ^{n},$ for every $X\in T_{p}M$.
\end{theorem}

We have a complete description of affine Szabó surfaces.

\begin{theorem}
\cite{DialloMassamba2017}\label{TheorSZA2} Let $\varSigma=(M,\nabla)$ be an
affine surface. Then $\varSigma$ is affine Szab\'o at $p\in M$ if and only
if the Ricci tensor of $(M,\nabla)$ is cyclic parallel at $p\in M$.
\end{theorem}

Next we investigate some particular case. The curvature of an affine surface
is encoded by its Ricci tensor. Fixing coordinates $(u_{1},u_{2})$ on $%
\varSigma$ and let $\nabla _{\partial _{i}}\partial _{j}=\Gamma
_{ij}^{k}\partial _{k}$ for $i,j,k=1,2$ where $\Gamma _{ij}^{k}=\Gamma
_{ij}^{k}(u_{1},u_{2})$. Then a straightforward calculation shows that the
components of the curvature tensor $\mathcal{R}$ are given by
\begin{equation}
\mathcal{R}(\partial _{1},\partial _{2})\partial _{1}=\rho _{21}\partial
_{1}-\rho _{11}\partial _{2},\;\;\mbox{and}\;\;\mathcal{R}(\partial
_{1},\partial _{2})\partial _{2}=\rho _{22}\partial _{1}-\rho _{12}\partial
_{2},  \notag
\end{equation}%
where $\rho _{ij},i,j=1,2$ are the components of the Ricci tensor given
\begin{align}
\rho _{21}& =\partial _{1}\Gamma _{12}^{1}-\partial _{2}\Gamma
_{11}^{1}+\Gamma _{12}^{1}\Gamma _{12}^{2}-\Gamma _{11}^{2}\Gamma _{22}^{1},
\notag  \label{Smalabcd} \\
\rho _{11}& =-\big(\partial _{1}\Gamma _{12}^{2}-\partial _{2}\Gamma
_{11}^{2}+\Gamma _{11}^{2}\Gamma _{12}^{1}+\Gamma _{12}^{2}\Gamma
_{12}^{2}-\Gamma _{11}^{1}\Gamma _{12}^{2}-\Gamma _{11}^{2}\Gamma _{22}^{2}%
\big),  \notag \\
\rho _{22}& =\partial _{1}\Gamma _{22}^{1}-\partial _{2}\Gamma
_{12}^{1}+\Gamma _{11}^{1}\Gamma _{22}^{1}+\Gamma _{12}^{1}\Gamma
_{22}^{2}-\Gamma _{12}^{1}\Gamma _{12}^{1}-\Gamma _{12}^{2}\Gamma _{22}^{1},
\notag \\
\rho _{12}& =-\big(\partial _{1}\Gamma _{22}^{2}-\partial _{2}\Gamma
_{12}^{2}+\Gamma _{11}^{2}\Gamma _{22}^{1}-\Gamma _{12}^{1}\Gamma _{12}^{2}%
\big).
\end{align}%
Let $X=\alpha _{1}\partial _{1}+\alpha _{2}\partial _{2}$ be a vector field
on $\varSigma$. It is easy check that the affine Szabó operator $\mathcal{S}%
(X)$ expresses, with respect to the basis $\{\partial _{1},\partial _{2}\}$,
as
\begin{equation}
(\mathcal{S}^{\nabla }(X))=\left(
\begin{array}{cc}
A & B \\
C & D%
\end{array}%
\right) .  \label{eq-mat}
\end{equation}%
where the coefficients $A$, $B$, $C$ and $D$ are given by
\begin{align*}
A& =\alpha _{1}^{2}\alpha _{2}\Big[\partial _{1}\rho _{21}-(\Gamma
_{11}^{1}+\Gamma _{12}^{2})\rho _{21}-\Gamma _{12}^{1}\rho _{11}-\Gamma
_{11}^{2}\rho _{22}\Big] \\
& +\alpha _{1}\alpha _{2}^{2}\Big[\partial _{2}\rho _{21}+\partial _{1}\rho
_{22}-(\Gamma _{12}^{1}+\Gamma _{22}^{2})\rho _{21}-(\rho _{12}+\rho
_{21})\Gamma _{12}^{1}-\Gamma _{22}^{1}\rho _{11}-3\Gamma _{12}^{2}\rho _{22}%
\Big] \\
& +\alpha _{2}^{3}\Big[\partial _{2}\rho _{22}-2\Gamma _{22}^{2}\rho
_{22}-(\rho _{12}+\rho _{21})\Gamma _{22}^{1}\Big], \\
B& =\alpha _{1}^{2}\alpha _{2}\Big[-\partial _{1}\rho _{11}+2\Gamma
_{11}^{1}\rho _{11}+(\rho _{12}+\rho _{21})\Gamma _{11}^{2}\Big] \\
& +\alpha _{1}\alpha _{2}^{2}\Big[-\partial _{2}\rho _{11}-\partial _{1}\rho
_{12}+3\Gamma _{12}^{1}\rho _{11}+\Gamma _{11}^{2}\rho _{22}+(\rho
_{12}+\rho _{21})\Gamma _{12}^{2}+(\Gamma _{11}^{1}+\Gamma _{12}^{2})\rho
_{12}\Big] \\
& +\alpha _{2}^{3}\Big[-\partial _{2}\rho _{12}+\Gamma _{22}^{1}\rho
_{11}+\Gamma _{12}^{2}\rho _{22}+(\Gamma _{12}^{1}+\Gamma _{22}^{2})\rho
_{12}\Big], \\
C& =\alpha _{1}^{3}\Big[-\partial _{1}\rho _{21}+(\Gamma _{11}^{1}+\Gamma
_{12}^{2})\rho _{21}+\Gamma _{12}^{1}\rho _{11}\Big] \\
& +\alpha _{1}^{2}\alpha _{2}\Big[-\partial _{2}\rho _{21}-\partial _{1}\rho
_{22}+(\Gamma _{12}^{1}+\Gamma _{22}^{2})\rho _{21}+\Gamma _{22}^{1}\rho
_{11}+3\Gamma _{12}^{2}\rho _{22}+(\rho _{12}+\rho _{21})\Gamma _{12}^{1}%
\Big] \\
& +\alpha _{1}\alpha _{2}^{2}\Big[-\partial _{2}\rho _{22}+2\Gamma
_{22}^{2}\rho _{22}+(\rho _{12}+\rho _{21})\Gamma _{22}^{1}\Big], \\
D& =\alpha _{1}^{3}\Big[\partial _{1}\rho _{11}-2\Gamma _{11}^{1}\rho
_{11}-(\rho _{12}+\rho _{21})\Gamma _{11}^{2}\Big] \\
& +\alpha _{1}^{2}\alpha _{2}\Big[\partial _{2}\rho _{11}+\partial _{1}\rho
_{12}-3\Gamma _{12}^{1}\rho _{11}-\Gamma _{11}^{2}\rho _{22}-(\Gamma
_{11}^{1}+\Gamma _{12}^{2})\rho _{12}-(\rho _{12}+\rho _{21})\Gamma _{12}^{2}%
\Big] \\
& +\alpha _{1}\alpha _{2}^{2}\Big[\partial _{2}\rho _{12}-\Gamma
_{22}^{1}\rho _{11}-\Gamma _{12}^{2}\rho _{22}-(\Gamma _{12}^{1}+\Gamma
_{22}^{2})\rho _{12}\Big].
\end{align*}

Its characteristic polynomial is given by
\begin{eqnarray*}
P_{\lambda} [\mathcal{S}^{\nabla} (X)]=\lambda^2 -\lambda(A+D) + (AD-BC).
\end{eqnarray*}

Here, we investigate affine surfaces whose Ricci tensor are skew-symmetric.

\begin{theorem}
\label{main1} Let $\nabla $ be an torsion-free affine connection on a
surface $\varSigma$. Then the Ricci tensor of $\nabla $ is skew-symmetric
and nonzero everywhere if and only if $(\varSigma,\nabla )$ is affine Szabó.
\end{theorem}

\begin{proof}
If the Ricci tensor of $\nabla $ is skew-symmetric, that is, $\rho
_{11}=\rho _{22}=0$ and $\rho _{12}=-\rho _{21}$. Then the Szabó operator is
nilpotent.

Conversely, if $(\varSigma,\nabla )$ is affine Szabó then the trace and
determinant of (\ref{eq-mat}) will be zero, which is possible only if $\rho
_{11}=\rho _{22}=0$ and $\rho _{12}=-\rho _{21}$.
\end{proof}

The investigation of affine connections with skew-symmetric Ricci tensor on
surfaces has been extremely attractive and fruitful over the recent years.
We refer to the paper \cite{Derdzinski2008} by Derdzinski for further
details. Taking into account the simplified Wong's theorem \cite[Th 4.2]%
{Wong} given in \cite{Derdzinski2008}, we have the following:

\begin{theorem}
\label{main2} If every point of an affine surface $\varSigma$ has a
neighborhood $U$ with coordinates $(u_{1},u_{2})$ in which the component
functions of a torsion-free affine connection $\nabla $ are $\Gamma
_{11}^{1}=-\partial _{1}\varphi $, $\Gamma _{22}^{2}=\partial _{2}\varphi $,
for some function $\varphi $, $\Gamma _{jk}^{l}=0$, unless $j=k=l$, then $(%
\varSigma,\nabla )$ is affine Szabó.
\end{theorem}

\begin{proof}
It easy to show that the Ricci tensor of $\nabla $ is skew-symmetric.
\end{proof}

A Lagrangian $L:U\rightarrow {\mathbb{R}}$ in a manifold $\varSigma$ is a
function on a nonempty open set $U\subset T\varSigma$. A Lagrangian $%
L:U\rightarrow{\mathbb{R}} $ gives rise to equations of motion, which are
the Euler-Lagrange equations, imposed on curves $t\rightarrow y(t)\in %
\varSigma$, the velocity $t\rightarrow v(t)\in T\varSigma$, lies entirely in
$U$. A fractional-linear function in a two-dimensional real vector space $%
\Pi $ is a rational function of the form $\alpha/\beta $, defined on a
nonempty open subset of $\Pi \backslash \ker \beta $, where $\alpha ,\beta
\in \Pi ^{\ast }$are linearly independent functionals. By using \cite[Th 11.1%
]{Derdzinski2008} and Theorem \ref{main1}, we have

\begin{theorem}
Let $\nabla $ be an torsion-free affine connection on a surface $\varSigma$.
If every point in $T\varSigma \backslash \varSigma$ has a neighborhood $U$
with a fractional-linear Lagrangian $L:U\rightarrow {\mathbb{R}}$ such that
the solutions of the Euler-Lagrange equations for $L$ coincide with those
geodesics of $\nabla $ which, lifted to $T\varSigma$, lie in $U$, then $(%
\varSigma,\nabla )$ is affine Szabó.
\end{theorem}

\begin{definition}
\textrm{\cite{Wong}} A tensor field $T$ is said to be recurrent if there
exists a $1$-form $\alpha $ such that $\nabla T$ = $\alpha \otimes T$, where
$\nabla $ is an affine connection. In particular, an affine surface $(%
\varSigma,\nabla )$ is said to be recurrent if its Ricci tensor is recurrent.
\end{definition}

\begin{theorem}
Let $(\varSigma,\nabla )$ be an affine Szabó surface. Then $(\varSigma%
,\nabla )$ is recurrent if and only if around each point there exists a
coordinate system $(U,u^{h})$ with the non-zero components of $\nabla $ are
\begin{equation*}
\Gamma _{11}^{1}=-\partial _{1}\varphi ,\quad \Gamma _{22}^{2}=\partial
_{2}\varphi
\end{equation*}%
for some scalar function $\varphi $ such that $\partial _{2}\partial
_{1}\varphi \not=0$. Moreover, $(\varSigma,\nabla )$ is not locally
symmetric.
\end{theorem}

\begin{proof}
Consider the Ricci tensor $\rho =\rho _{a}+\rho _{s}$, where $\rho _{a}$ is
the antisymmetric part of $\rho $ and $\rho _{s}$ is the symmetric part of $%
\rho $. Then by using Theorem \ref{main1}, we can say that $(\varSigma%
,\nabla )$ is an affine Szabó if and only if the Ricci tensor of $\nabla $
is skew-symmetric and nonzero everywhere. Then it follows from \cite[Th 4.2]%
{Wong} that one of the three possibility for a non-flat recurrent affine
surface is the one in which around each point there exists a coordinate
system $(U,u^{h})$ with the non-zero components of $\nabla $ are
\begin{equation*}
\Gamma _{11}^{1}=-\partial _{1}\varphi ,\quad \Gamma _{22}^{2}=\partial
_{2}\varphi
\end{equation*}%
for some scalar function $\varphi $ such that $\partial _{2}\partial
_{1}\varphi \not=0$. Now, it is easy to calculate that $\rho _{21}=-\rho
_{12}=\partial _{2}\partial _{1}\varphi $, which is never zero. So, $(%
\varSigma,\nabla )$ is not locally symmetric.
\end{proof}

By using the result of \cite[Th 2.2]{Wong} and Theorem \ref{main1}, we can
say that

\begin{theorem}
Let $(\varSigma,\nabla )$ be an affine Szabó recurrent surface. Then the
recurrence covector of a recurrence tensor is not locally a gradient.
\end{theorem}

\section{The deformed Riemannian extensions of an affine Szab\'o manifold}

\label{Main}

A pseudo-Riemannian manifold $(M,g)$ is said to be Szabó if the Szabó
operators $\mathcal{S}(X)=(\nabla _{X}R)(\cdot ,X)X$ has constant
eigenvalues on the unit pseudo-sphere bundles $S^{\pm }(TM)$. Any Szabó
manifold is locally symmetric in the Riemannian \cite{Szabo} and the
Lorentzian \cite{GilkeyStravrov2002} setting but the higher signature case
supports examples with nilpotent Szabó operators (cf. \cite%
{GilkeyIvanovaZhang2003} and the references therein). Now, we will prove the
following result:

\begin{theorem}
\label{main3} Let $(M,\nabla )$ be a $2$-dimensional smooth torsion-free
affine manifold. Then the following assertions are equivalent:

\begin{enumerate}
\item $(M,\nabla)$ is an affine Szab\'o manifold.

\item The deformed Riemannian extension $(T^*M,g_{(\nabla,\phi)})$ of $%
(M,\nabla)$ is a pseudo-Riemannian nilpotent Szab\'o manifold of neutral
signature.
\end{enumerate}
\end{theorem}

\begin{proof}
Let $\Gamma _{ij}^{k}$ be the coefficients of the torsion free affine
connection $\nabla $ and $\phi _{ij}$ denote the local components of $\phi $%
. Then, the deformed Riemannian extension of the torsion free affine
connection $\nabla $ is the pseudo-Riemannian metric tensor on $T^{\ast }M$
of signature $(2,2)$ given by
\begin{eqnarray}
g_{(\nabla ,\phi )} &=&\Big(\phi _{11}(u_{1},u_{2})-2u_{3}\Gamma
_{11}^{1}-2u_{4}\Gamma _{11}^{2}\Big)du_{1}\otimes du_{1}  \notag
\label{eq4.1} \\
&&+\Big(\phi _{22}(u_{1},u_{2})-2u_{3}\Gamma _{22}^{1}-2u_{4}\Gamma _{22}^{2}%
\Big)du_{2}\otimes du_{2}  \notag \\
&&+\Big(\phi _{12}(u_{1},u_{2})-2u_{3}\Gamma _{12}^{1}-2u_{4}\Gamma _{12}^{2}%
\Big)(du_{1}\otimes du_{2}+du_{1}\otimes du_{2})  \notag \\
&&+(du_{1}\otimes du_{3}+du_{3}\otimes du_{1}+du_{2}\otimes
du_{4}+du_{4}\otimes du_{2}).
\end{eqnarray}%
A straightforward calculation shows that the non-zero Christoffel symbols $%
\tilde{\Gamma}_{\alpha \beta }^{\gamma }$ of the Levi-Civita connection are
given as follows
\begin{eqnarray*}
\tilde{\Gamma}_{ij}^{k} &=&\Gamma _{ij}^{k};\;\tilde{\Gamma}_{i^{\prime
}j}^{k^{\prime }}=-\Gamma _{jk}^{i};\;\tilde{\Gamma}_{ij^{\prime
}}^{k^{\prime }}=-\Gamma _{ij}^{j} \\
\tilde{\Gamma}_{ij}^{k^{\prime }} &=&\sum_{r=1}^{2}\Big(\partial _{k}\Gamma
_{ij}^{r}-\partial _{i}\Gamma _{jk}^{r}-\partial _{j}\Gamma
_{ik}^{r}+2\sum_{l=1}^{2}\Gamma _{kl}^{r}\Gamma _{ij}^{l}\Big) \\
&&+\frac{1}{2}\Big(\partial _{i}\phi _{jk}+\partial _{j}\phi _{ik}-\partial
_{k}\phi _{ij}\Big)-\sum_{l=1}^{2}\phi _{kl}\Gamma _{ij}^{l}
\end{eqnarray*}%
where $(i,j,k,l,r=1,2)$ and $(i^{\prime }=i+2,j^{\prime }=j+2,k^{\prime
}=k+2,r^{\prime }=r+2)$. The non-zero components of the curvature tensor of $%
(T^{\ast }M,g_{(\nabla ,\phi )})$ up to the usual symmetries are given as
follows (we omit $\widetilde{R}_{kji}^{h^{\prime }}$, as it plays no role in
our considerations)
\begin{equation*}
\widetilde{R}_{kji}^{h}=R_{kji}^{h},\;\;\widetilde{R}_{kji}^{h^{\prime
}},\;\;\widetilde{R}_{kji^{\prime }}^{h^{\prime }}=-R_{kjh}^{i},\;\;%
\widetilde{R}_{k^{\prime }ji}^{h^{\prime }}=R_{hij}^{k},
\end{equation*}%
where $R_{kji}^{h}$ are the components of the curvature tensor of $(M,\nabla
)$. (For more details, see \cite{Calvino2010}.)\newline

Let $\tilde{X}=\alpha _{i}\partial _{i}+\alpha _{i^{\prime }}\partial
_{i^{\prime }}$ be a vector field on $T^{\ast }M$. Then the matrix of the
Szabó operator $\tilde{S}(\tilde{X})$ with respect to the basis $\{\partial
_{i},\partial _{i^{\prime }}\}$ is of the form
\begin{equation*}
\tilde{\mathcal{S}}(\tilde{X})=\left(
\begin{array}{cc}
\mathcal{S}^{\nabla }(X) & 0 \\*
& {}^{t}\mathcal{S}^{\nabla }(X)%
\end{array}%
\right) .
\end{equation*}%
where $\mathcal{S}^{\nabla }(X)$ is the matrix of the affine Szabó operator
on $M$ relative to the basis $\{\partial _{i}\}$. Note that the
characteristic polynomial $P_{\lambda }[\tilde{\mathcal{S}}(\tilde{X})]$ of $%
\tilde{\mathcal{S}}(\tilde{X})$ and $P_{\lambda }[\mathcal{S}^{\nabla }(X)]$
of $\mathcal{S}^{\nabla }(X)$ are related by
\begin{equation*}
P_{\lambda }[\tilde{\mathcal{S}}(\tilde{X})]=P_{\lambda }[\mathcal{S}%
^{\nabla }(X)]\cdot P_{\lambda }[{}^{t}\mathcal{S}^{\nabla }(X)].
\end{equation*}%
Now, if the affine manifold $(M,\nabla )$ is assumed to be affine Szabó,
then $\mathcal{S}^{\nabla }(X)$ has zero eigenvalues for each vector field $%
X $ on $M$. Therefore, it follows from (\ref{eq3.2}) that the eigenvalues of
$\tilde{\mathcal{S}}(\tilde{X})$ vanish for every vector field $\tilde{X}$
on $T^{\ast }M$. Thus $(T^{\ast }M,g_{\nabla })$ is pseudo-Riemannian Szabó
manifold.

Conversely, assume that $(T^{\ast }M,g_{\nabla })$ is an pseudo-Riemannian
Szabó manifold. If $X=\alpha _{i}\partial _{i}$ is an arbitrary vector field
on $M$, then $\tilde{X}=\alpha _{i}\partial _{i}+\frac{1}{2\alpha _{i}}%
\partial _{i^{\prime }}$ is an unit vector field at every point of the zero
section on $T^{\ast }M$. Then from (\ref{eq4.2}), we see that, the
characteristic polynomial $P_{\lambda }[\tilde{\mathcal{S}}(\tilde{X})]$ of $%
\tilde{\mathcal{S}}(\tilde{X})$ is the square of the characteristic
polynomial $P_{\lambda }[\mathcal{S}^{\nabla }(X)]$ of $\mathcal{S}^{\nabla
}(X)$. Since for every unit vector field $\tilde{X}$ on $T^{\ast }M$ the
characteristic polynomial $P_{\lambda }[\tilde{\mathcal{S}}(\tilde{X})]$
should be the same, it follows that for every vector field $X$ on $M$ the
characteristic polynomial $P_{\lambda }[\mathcal{S}^{\nabla }(X)]$ is the
same. Hence $(M,\nabla )$ is affine Szabó.
\end{proof}

For an example, we have the following:

\begin{theorem}
\cite{DialloGupta2020} Let $M=\mathbb{R}^{2}$ and $\nabla $ be the torsion
free connection defined by $\nabla _{\partial _{1}}\partial
_{1}=f_{1}(u_{1})\partial _{2}$ and $\nabla _{\partial _{1}}\partial
_{2}=f_{2}(u_{1})\partial _{2}$. Assume that $f_{1}$ and $f_{2}$ satisfies $%
\partial _{1}b=0$ and $\partial _{2}b=0,$ where $b=\partial
_{1}f_{2}+f_{2}^{2}$. Then the pseudo-Riemannian metric $g_{(\nabla ,\phi )}$
on the cotangent bundle $T^{\ast }M$ of neutral signature $(2,2)$ defined by
setting
\begin{eqnarray*}
g_{(\nabla,\phi)} &=&(\phi _{11}-2u_{4}f_{1})du_{1}\otimes du_{1}+\phi
_{22}du_{2}\otimes du_{2}  \notag \\
&&+(\phi _{12}-2u_{4}f_{2})(du_{1}\otimes du_{2}+du_{1}\otimes du_{2})
\notag \\
&&+(du_{1}\otimes du_{3}+du_{3}\otimes du_{1}+du_{2}\otimes
du_{4}+du_{4}\otimes du_{2}).
\end{eqnarray*}%
is Szab\'o for any symmetric $(0,2)$-tensor field $\phi $.
\end{theorem}

\end{document}